\documentclass[11pt]{article}
\usepackage{paper_options}

\newcommand{\for}{\text{for }}

\usepackage{tikz}
\usepackage{pgfplots}\pgfplotsset{compat=1.13}
\usetikzlibrary{calc}
\usepackage{graphics}
\usepackage{xcolor}

\colorlet{pink}{red!75}
\colorlet{sky}{blue!60}

\title{Long monochromatic paths and cycles in 2-colored bipartite graphs} 
\author{Louis DeBiasio$^{1}$, Robert A. Krueger$^{2}$}
\date{\today}

\begin{document}

\maketitle
\noindent\footnotetext[1]{Department of Mathematics, Miami University {\tt debiasld@miamioh.edu}. Research supported in part by Simons Foundation Collaboration Grant \# 283194.}
\noindent\footnotetext[2]{Department of Mathematics, Miami University {\tt kruegera@miamioh.edu}.}

\begin{abstract}
Gy\'arf\'as and Lehel and independently Faudree and Schelp proved that in any 2-coloring of the edges of $K_{n,n}$ there exists a monochromatic path on at least $2\ceiling{n/2}$ vertices, and this is tight.  We prove a stability version of this result which holds even if the host graph is not complete; that is, if $G$ is a balanced bipartite graph on $2n$ vertices with minimum degree at least $(3/4+o(1))n$, then in every 2-coloring of the edges of $G$, either there exists a monochromatic cycle on at least $(1+o(1))n$ vertices, or the coloring of $G$ is close to an extremal coloring -- in which case $G$ has a monochromatic path on at least $2\ceiling{n/2}$ vertices and a monochromatic cycle on at least $2\floor{n/2}$ vertices.  Furthermore, we determine an asymptotically tight bound on the length of a longest monochromatic cycle in a 2-colored balanced bipartite graph on $2n$ vertices with minimum degree $\delta n$ for all $0\leq \delta\leq 1$.  
\end{abstract}

\section{Introduction}

The order of a path $P$ is the number of vertices in $P$ and the length of $P$ is the number of edges in $P$.  


The following is a classical result in graph-Ramsey theory. 

\begin{theorem}[Gerencs\'er, Gy\'arf\'as \cite{GG}]\label{thm:GG}
Let $n$ be a positive integer.  In every 2-coloring of the edges of $K_n$, there exists a monochromatic path on at least $\ceiling{(2n+1)/3}$ vertices. Furthermore, there exists a 2-coloring of the edges of $K_n$ such that the longest monochromatic path has $\ceiling{(2n+1)/3}$ vertices.
\end{theorem}

Theorem \ref{thm:GG} has been extended in a number of ways.  Gy\'arf\'as, S\'ark\"ozy, and Szemer\'edi \cite{GSS} proved a stability version; that is, for all $\alpha>0$, there exists $n_0$ and $\eta>0$ such that for all $n\geq n_0$, in any 2-coloring of the edges of $K_n$, either there is a monochromatic path on at least $(2/3+\eta)n$ vertices or the 2-coloring is close to an extremal coloring with parameter $\alpha$.  Gy\'arf\'as and S\'ark\"ozy \cite{GS} then proved a version of Theorem \ref{thm:GG} for non-complete graphs; that is, they proved that if $G$ is a graph on $n$ vertices with $\delta(G)\geq (3/4+o(1))n$, then in every 2-coloring of the edges of $G$, there exists a monochromatic path on at least $(2/3-o(1))n$ vertices and they note that there exists graphs with minimum degree $3n/4-1$ and 2-colorings of such graphs where the longest monochromatic path has at most $n/2$ vertices.  Later this result was strengthened by Benevides, \L uczak, Skokan, Scott, and White \cite{BLSSW} who proved that if $\delta(G)\geq 3n/4$, then either $G$ contains a monochromatic path on at least $(2/3+o(1))n$ vertices or $G$ is close to an extremal example in which case $G$ contains a path on at least $(2/3-o(1))n$ vertices (in fact their result says something more about cycles of specific lengths).  Finally, White \cite{W} determined an asymptotically tight bound on the length of a longest path in an arbitrary 2-coloring of a graph $G$ with minimum degree at least $\delta n$ for all $0\leq \delta \leq 3/4$. 

\subsection{Bipartite graphs}

Gy\'arf\'as and Lehel \cite{GL} and independently Faudree and Schelp \cite{FS} proved a bipartite analog of Theorem \ref{thm:GG}.  

\begin{theorem}\label{thm:GL}
Let $n$ be a positive integer.  In every 2-coloring of the edges of $K_{n,n}$ there exists a monochromatic path on at least $2\ceiling{n/2}$ vertices.  Furthermore, this is best possible by Example \ref{exLargeDeg}.
\end{theorem}

The purpose of this paper is to extend Theorem \ref{thm:GL} in a few ways, analogous to the above extensions of Theorem \ref{thm:GG}.

Our first result gives an asymptotically tight bound on the length of a monochromatic cycle in a balanced bipartite graph on $2n$ vertices with minimum degree $\delta n$ for all $0\leq \delta \leq 1$ (see Figure \ref{degreeplot}).  

\begin{theorem}\label{thm2}
For all $0\leq \delta \leq 1$ and $\ep>0$, there exists $n_0$ such that if $G$ is a balanced bipartite graph on $2n\geq 2n_0$ vertices with $\delta(G)\geq \delta n$, then in every 2-coloring of the edges of $G$ there exists a monochromatic cycle of order at least $(f(\delta) - \ep) n$, where
\[ f(\delta) = 
\begin{cases}
      \delta,& \for 0\leq \delta \leq 2/3, \\
      4\delta-2,& \for 2/3\leq \delta\leq 3/4, \\
      1,& \for 3/4\leq \delta\leq 1.
    \end{cases}
\]
\end{theorem}

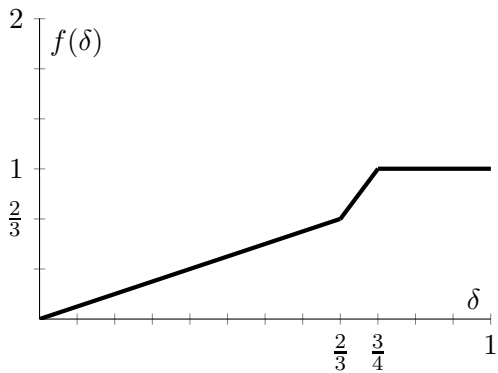
\begin{figure}[ht]
\centering
\begin{tikzpicture}
\begin{axis}[axis x line=middle, axis y line=middle, x axis line style=-, y axis line style=-, xtick={0,.0833,...,1.0829}, ytick={0,.333,...,1.998},
xmin=0, xmax=1, ymin=0, ymax=2, xticklabels={,,,,,,,,$\frac{2}{3}$,$\frac{3}{4}$,,,$1$}, yticklabels={,,$\frac{2}{3}$,$1$,,,$2$}, x=6cm, y=2cm, xlabel=$\delta$, ylabel=$f(\delta)$, no markers, every axis plot/.append style={ultra thick}]
\addplot[domain=0:.666]{x};
\addplot[domain=.666:.75]{4*x-2};
\addplot[domain=.75:1]{1};
\end{axis}
\end{tikzpicture}
\caption{Lower bound on the length of a longest monochromatic cycle as a function of degree}\label{degreeplot}
\end{figure}

Our next result simultaneously gives a stability version and a large minimum degree version of Theorem \ref{thm:GL}.  Before stating the result, we define an what we mean by an extremal coloring.

\begin{definition}[$\eta$-extremal coloring]
Let $\eta \geq 0$ and let $G$ be a balanced bipartite graph on $2n$ vertices.  Let $c:E(G)\to \{1,2\}$ be a 2-coloring of the edges of $G$.  We say that $c$ is an $\eta$-extremal coloring if there exists $X'\subseteq X$ and a partition $\{Y_1, Y_2\}$ of $Y$ such that $|X'|, |Y_1|, |Y_2|\geq (1/2-\eta)n$, $e_R(X', Y_1)\leq \eta n^2$, and $e_B(X',Y_2)\leq \eta n^2$.  
\end{definition}

\begin{theorem}\label{thm1}
For all real numbers $\gamma, \eta$ with $0 \leq 64\sqrt{\eta} < \gamma \leq \frac{1}{4}$, there exists $n_0$ such that if $G$ is a balanced bipartite graph on $2n\geq 2n_0$ vertices with $\delta(G)\geq (3/4+\gamma)n$, then in every 2-coloring of $G$, either there exists a monochromatic cycle on at least $(1+\eta)n$ vertices or the edge coloring is $\eta$-extremal, in which case there exists a monochromatic path on at least $2\ceiling{n/2}$ vertices and a monochromatic cycle on at least $2\floor{n/2}$ vertices.
\end{theorem}

While it is likely that the following result could be obtained from the arguments in \cite{GL} and \cite{FS} (for all $n$), it doesn't seem to appear in the literature, so we state it here explicitly.

\begin{corollary}\label{cor1}
Let $n$ sufficiently large positive integer. In every 2-coloring of the edges of $K_{n,n}$ there exists a monochromatic cycle on at least $2\floor{n/2}$ vertices.  Furthermore, this is best possible by Examples \ref{exLargeDeg} and \ref{exCycle}.
\end{corollary}

In Section \ref{sec:examples}, we give examples to show that the degree conditions in Theorems \ref{thm2} and \ref{thm1} are asymptotically best possible.  It would be interesting to obtain exact versions of these results.

The proof of Theorems \ref{thm2} and \ref{thm1} go roughly as follows: we first find large monochromatic connected components (monochromatic component, for short) in Section \ref{sec:comps}; using that we find large monochromatic connected matchings in Section \ref{sec:matching} and prove a stability result regarding monochromatic connected matchings which will be needed for the proof of Theorem \ref{thm1}; the regularity lemma is introduced in Section \ref{sec:reg} along with the machinery needed to transfer results about monochromatic connected matchings to monochromatic cycles; we then combine all of these tools to prove Theorem \ref{thm2} in Section \ref{sec:circ} and Theorem \ref{thm1} in Sections \ref{sec:stability} and \ref{sec:extremal}.

\subsection{Notation}

We say that $G$ is an \emph{$X,Y$-bipartite graph} if $X$ and $Y$ are independent sets in $G$ and $\{X,Y\}$ forms a partition of $V(G)$.  Given an $X,Y$-bipartite graph $G$ and $X'\subseteq X$ and $Y'\subseteq Y$, we write $[X',Y']$ to the bipartite graph induced by $X'\cup Y'$.  Throughout the paper, when we refer to a 2-coloring of the edges of $G$, we mean $E(G)=E_R\cup E_B$ where $E_R$ are the set of red edges and $E_B$ are the set of blue edges.  Note that we do not require $E_R\cap E_B=\emptyset$, i.e. a 2-coloring of the edges always means a 2-multi-coloring of the edges.  This slightly more general setting is convenient because when we apply the regularity lemma, we will have edges which are $\ep$-regular in both red and blue and we would like to retain this information for the purposes of proving a stability result.  Of course any result which applies to every 2-multi-coloring of the edges applies to every 2-coloring of the edges, so we lose nothing by making this choice.  We write $N_R(v)=\{u: uv\in E_R(G)\}$ and $N_B(v)=\{u: uv\in E_B(G)\}$ and refer the vertices in $N_R(v)$ as the \emph{red neighbors} of $v$ and the vertices in $N_B(v)$ as the \emph{blue neighbors} of $v$.  Similarly, for $X' \subseteq X$ and $Y' \subseteq Y$, we write $[X',Y']_R$ ($[X',Y']_B$) for the spanning subgraph of red (blue) edges of $[X',Y']$. We also write $\delta_R(X',Y') = \min_{x \in X'} d_R(x, Y')$ for the \emph{minimum red degree from $X'$ to $Y'$} (similarly for blue). Note that in general, $\delta_R(X', Y') \neq \delta_R(Y', X')$.

We also use the standard notation $[k] = \{1,\dots,k\}$.

\section{Examples}\label{sec:examples}

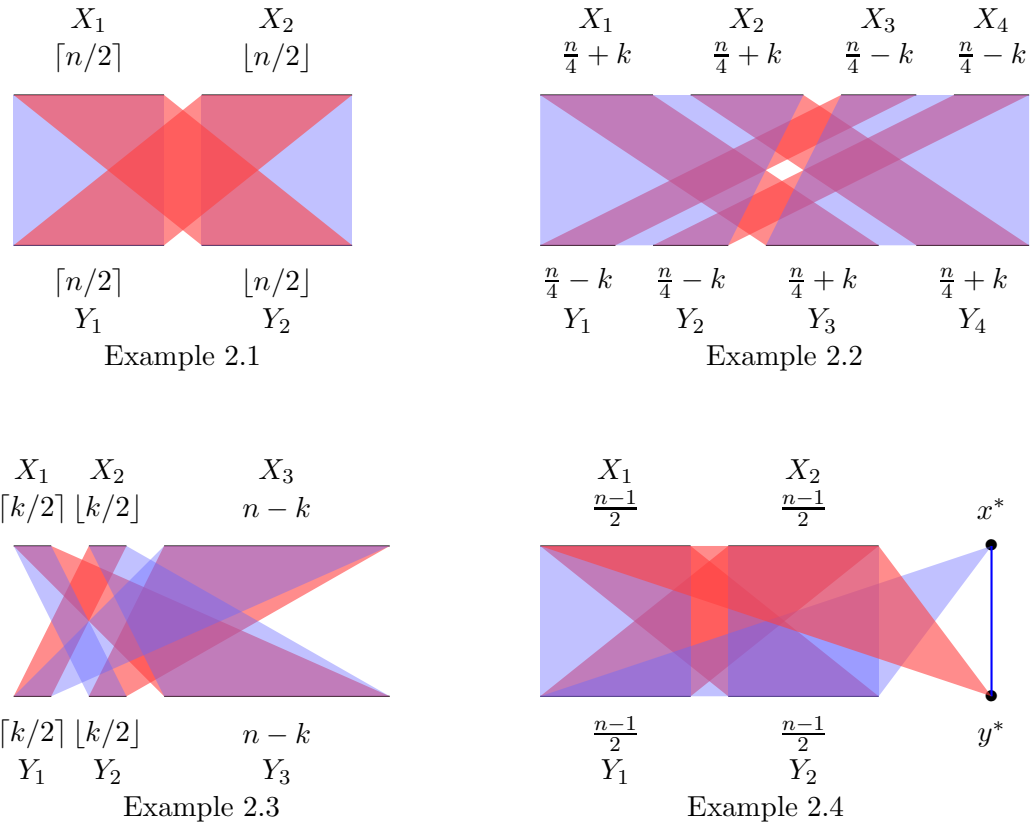
\begin{figure}[ht]
\centering
\begin{tikzpicture}
\begin{scope}[shift={(0,0)}]
  \draw (0,0) -- (2,0);
  \draw (2.5,0) -- (4.5,0);

  \draw (0,2) -- (2,2);
  \draw (2.5,2) -- (4.5,2);

  \draw[fill=sky, draw opacity=0, fill opacity=.4] (0,0) -- (2,0) -- (2,2) -- (0,2) -- (0,0);
  \draw[fill=sky, draw opacity=0, fill opacity=.4] (2.5,0) -- (4.5,0) -- (4.5,2) -- (2.5,2) -- (2.5,0);
  \draw[fill=pink, draw opacity=0, fill opacity=.6] (0,0) -- (2,0) -- (4.5,2) -- (2.5,2) -- (0,0);
  \draw[fill=pink, draw opacity=0, fill opacity=.6] (2.5,0) -- (4.5,0) -- (2,2) -- (0,2) -- (2.5,0);

  \draw (1,-.5) node{$\lceil n/2 \rceil$};
  \draw (3.5,-.5) node{$\lfloor n/2 \rfloor$};
  \draw (1,2.5) node{$\lceil n/2 \rceil$};
  \draw (3.5,2.5) node{$\lfloor n/2 \rfloor$};
  \draw (1,-1) node{$Y_1$};
  \draw (3.5,-1) node{$Y_2$};
  \draw (1,3) node{$X_1$};
  \draw (3.5,3) node{$X_2$};

  \draw (2.25,-1.5) node{Example \ref{exLargeDeg}};
\end{scope}

\begin{scope}[shift={(7,0)}]
  \coordinate (X1A) at (0,2);
  \coordinate (X1B) at (1.5,2);
  \coordinate (X2A) at (2,2);
  \coordinate (X2B) at (3.5,2);
  \coordinate (X3A) at (4,2);
  \coordinate (X3B) at (5,2);
  \coordinate (X4A) at (5.5,2);
  \coordinate (X4B) at (6.5,2);

  \coordinate (Y1A) at (0,0);
  \coordinate (Y1B) at (1,0);
  \coordinate (Y2A) at (1.5,0);
  \coordinate (Y2B) at (2.5,0);
  \coordinate (Y3A) at (3,0);
  \coordinate (Y3B) at (4.5,0);
  \coordinate (Y4A) at (5,0);
  \coordinate (Y4B) at (6.5,0);

  \coordinate (d) at (0,.5);

  \draw (X1A) -- (X1B);
  \draw (X2A) -- (X2B);
  \draw (X3A) -- (X3B);
  \draw (X4A) -- (X4B);
  \draw (Y1A) -- (Y1B);
  \draw (Y2A) -- (Y2B);
  \draw (Y3A) -- (Y3B);
  \draw (Y4A) -- (Y4B);

  \draw ($(X1A)!.5!(X1B)+2*(d)$) node{$X_1$};
  \draw ($(X1A)!.5!(X1B)+(d)$) node{$\frac{n}{4}+k$};
  \draw ($(X2A)!.5!(X2B)+2*(d)$) node{$X_2$};
  \draw ($(X2A)!.5!(X2B)+(d)$) node{$\frac{n}{4}+k$};
  \draw ($(X3A)!.5!(X3B)+2*(d)$) node{$X_3$};
  \draw ($(X3A)!.5!(X3B)+(d)$) node{$\frac{n}{4}-k$};
  \draw ($(X4A)!.5!(X4B)+2*(d)$) node{$X_4$};
  \draw ($(X4A)!.5!(X4B)+(d)$) node{$\frac{n}{4}-k$};
  \draw ($(Y1A)!.5!(Y1B)-2*(d)$) node{$Y_1$};
  \draw ($(Y1A)!.5!(Y1B)-(d)$) node{$\frac{n}{4}-k$};
  \draw ($(Y2A)!.5!(Y2B)-2*(d)$) node{$Y_2$};
  \draw ($(Y2A)!.5!(Y2B)-(d)$) node{$\frac{n}{4}-k$};
  \draw ($(Y3A)!.5!(Y3B)-2*(d)$) node{$Y_3$};
  \draw ($(Y3A)!.5!(Y3B)-(d)$) node{$\frac{n}{4}+k$};
  \draw ($(Y4A)!.5!(Y4B)-2*(d)$) node{$Y_4$};
  \draw ($(Y4A)!.5!(Y4B)-(d)$) node{$\frac{n}{4}+k$};

  \draw[fill=pink, draw opacity=0, fill opacity=.6] (X1A) -- (X1B) -- (Y3B) -- (Y3A) -- (X1A);
  \draw[fill=pink, draw opacity=0, fill opacity=.6] (X2A) -- (X2B) -- (Y4B) -- (Y4A) -- (X2A);
  \draw[fill=pink, draw opacity=0, fill opacity=.6] (X3A) -- (X3B) -- (Y1B) -- (Y1A) -- (X3A);
  \draw[fill=pink, draw opacity=0, fill opacity=.6] (X4A) -- (X4B) -- (Y2B) -- (Y2A) -- (X4A);

  \draw[fill=sky, draw opacity=0, fill opacity=.4] (X1A) -- (X2B) -- (Y2B) -- (Y1A) -- (X1A);
  \draw[fill=sky, draw opacity=0, fill opacity=.4] (X3A) -- (X4B) -- (Y4B) -- (Y3A) -- (X3A);

  \draw ($(Y1A)!.5!(Y4B)-3*(d)$) node{Example \ref{exMediumDeg}};
\end{scope}

\begin{scope}[shift={(0,-6)}]
  \coordinate (X1A) at (0,2);
  \coordinate (X1B) at (.5,2);
  \coordinate (X2A) at (1,2);
  \coordinate (X2B) at (1.5,2);
  \coordinate (X3A) at (2,2);
  \coordinate (X3B) at (5,2);

  \coordinate (Y1A) at (0,0);
  \coordinate (Y1B) at (.5,0);
  \coordinate (Y2A) at (1,0);
  \coordinate (Y2B) at (1.5,0);
  \coordinate (Y3A) at (2,0);
  \coordinate (Y3B) at (5,0);

  \coordinate (d) at (0,.5);

  \draw (X1A) -- (X1B);
  \draw (X2A) -- (X2B);
  \draw (X3A) -- (X3B);
  \draw (Y1A) -- (Y1B);
  \draw (Y2A) -- (Y2B);
  \draw (Y3A) -- (Y3B);

  \draw ($(X1A)!.5!(X1B)+2*(d)$) node{$X_1$};
  \draw ($(X1A)!.5!(X1B)+(d)$) node{$\lceil k/2 \rceil$};
  \draw ($(X2A)!.5!(X2B)+2*(d)$) node{$X_2$};
  \draw ($(X2A)!.5!(X2B)+(d)$) node{$\lfloor k/2 \rfloor$};
  \draw ($(X3A)!.5!(X3B)+2*(d)$) node{$X_3$};
  \draw ($(X3A)!.5!(X3B)+(d)$) node{$n-k$};
  \draw ($(Y1A)!.5!(Y1B)-2*(d)$) node{$Y_1$};
  \draw ($(Y1A)!.5!(Y1B)-(d)$) node{$\lceil k/2 \rceil$};
  \draw ($(Y2A)!.5!(Y2B)-2*(d)$) node{$Y_2$};
  \draw ($(Y2A)!.5!(Y2B)-(d)$) node{$\lfloor k/2 \rfloor$};
  \draw ($(Y3A)!.5!(Y3B)-2*(d)$) node{$Y_3$};
  \draw ($(Y3A)!.5!(Y3B)-(d)$) node{$n-k$};

  \draw[fill=pink, draw opacity=0, fill opacity=.6] (X1A) -- (X1B) -- (Y3B) -- (Y3A) -- (X1A);
  \draw[fill=pink, draw opacity=0, fill opacity=.6] (X2A) -- (X2B) -- (Y1B) -- (Y1A) -- (X2A);
  \draw[fill=pink, draw opacity=0, fill opacity=.6] (X3A) -- (X3B) -- (Y2B) -- (Y2A) -- (X3A);

  \draw[fill=sky, draw opacity=0, fill opacity=.4] (X1A) -- (X1B) -- (Y2B) -- (Y2A) -- (X1A);
  \draw[fill=sky, draw opacity=0, fill opacity=.4] (X2A) -- (X2B) -- (Y3B) -- (Y3A) -- (X2A);
  \draw[fill=sky, draw opacity=0, fill opacity=.4] (X3A) -- (X3B) -- (Y1B) -- (Y1A) -- (X3A);

  \draw ($(Y1A)!.5!(Y3B)-3*(d)$) node{Example \ref{exSmallDeg}};
\end{scope}

\begin{scope}[shift={(7,-6)}]
  \coordinate (X1A) at (0,2);
  \coordinate (X1B) at (2,2);
  \coordinate (X2A) at (2.5,2);
  \coordinate (X2B) at (4.5,2);

  \coordinate (Y1A) at (0,0);
  \coordinate (Y1B) at (2,0);
  \coordinate (Y2A) at (2.5,0);
  \coordinate (Y2B) at (4.5,0);

  \coordinate (Xmid) at (6,2);
  \coordinate (Ymid) at (6,0);

  \draw (X1A) -- (X1B);
  \draw (Xmid) node{\textbullet};
  \draw (X2A) -- (X2B);

  \draw (Y1A) -- (Y1B);
  \draw (Ymid) node{\textbullet};
  \draw (Y2A) -- (Y2B);

  \draw[fill=pink, draw opacity=0, fill opacity=.6] (Y1A) -- (Y1B) -- (X2B) -- (X2A) -- (Y1A);
  \draw[fill=pink, draw opacity=0, fill opacity=.6] (Y2A) -- (Y2B) -- (X1B) -- (X1A) -- (Y2A);
  \draw[fill=sky, draw opacity=0, fill opacity=.4] (Y1A) -- (Y1B) -- (X1B) -- (X1A) -- (Y1A);
  \draw[fill=sky, draw opacity=0, fill opacity=.4] (Y2A) -- (Y2B) -- (X2B) -- (X2A) -- (Y2A);

  \draw (1,-.5) node{$\frac{n-1}{2}$};
  \draw (3.5,-.5) node{$\frac{n-1}{2}$};
  \draw (1,2.5) node{$\frac{n-1}{2}$};
  \draw (3.5,2.5) node{$\frac{n-1}{2}$};
  \draw (1,-1) node{$Y_1$};
  \draw (3.5,-1) node{$Y_2$};
  \draw (1,3) node{$X_1$};
  \draw (3.5,3) node{$X_2$};
  \draw ($(Ymid)-(0,.5)$) node{$y^*$};
  \draw ($(Xmid)+(0,.5)$) node{$x^*$};

  \draw[fill=sky, draw opacity=0, fill opacity=.4] (Xmid) -- (Y1A) -- (Y2B) -- (Xmid);
  \draw[draw=blue, thick] (Ymid) -- (Xmid);
  \draw[fill=pink, draw opacity=0, fill opacity=.6] (Ymid) -- (X1A) -- (X2B) -- (Ymid);

  \draw ($(Y1A)!.5!(Ymid)-3*(0,.5)$) node{Example \ref{exCycle}};
\end{scope}
\end{tikzpicture}
\caption{Tightness Examples}
\label{tightnessexamples}
\end{figure}

The following three examples show the asymptotic tightness of Theorem \ref{thm2} in the three different regimes.

\begin{example}\label{exLargeDeg}
Let $K_{n,n}$ be an $X,Y$-bipartite graph, and partition $X$ and $Y$ into $X_1, X_2$ and $Y_1, Y_2$, respectively, where $|X_1| = |Y_1| = \lceil n/2 \rceil$ and $|X_2| = |Y_2| = \lfloor n/2 \rfloor$. Color $[X_1, Y_1]$ and $[X_2, Y_2]$ blue, while coloring $[X_1, Y_2]$ and $[X_2, Y_1]$ red. The order of the longest monochromatic path (and cycle) is $2\lceil n/2 \rceil$.
\end{example}

\begin{example}\label{exMediumDeg}
Suppose $n$ is divisible by $4$. Let $0\leq k\leq n/12$. Let $G$ be a balanced $X,Y$-bipartite graph on $2n$ vertices defined as follows. Partition $X$ and $Y$ into $X_1, X_2, X_3, X_4$ and $Y_1, Y_2, Y_3, Y_4$ such that $|X_1| = |X_2| = |Y_3| = |Y_4| = n/4 + k$ and $|X_3| = |X_4| = |Y_1| = |Y_2| = n/4 - k$. Include all the edges in the following bipartite subgraphs (and no others) and color them as indicated: color $[X_1 \cup X_2, Y_1\cup  Y_2]$ and $[X_3 \cup X_4, Y_3 \cup Y_4]$ blue and color $[X_1, Y_3]$, $[X_2, Y_4]$, $[X_3, Y_1]$, and $[X_4, Y_2]$ red. Since $k \leq n/12$, the order of the longest monochromatic path (and cycle) in $G$ is $n - 4k$, while the minimum degree of $G$ is $3n/4 - k$.
\end{example}

\begin{example}\label{exSmallDeg}
Let $0\leq k\leq \frac{n}{3}$. Let $G$ be a balanced $X,Y$-bipartite graph on $2n$ vertices defined as follows. Partition $X$ and $Y$ into $X_1, X_2, X_3$ and $Y_1, Y_2, Y_3$ such that $|X_1| = |Y_1| = \lceil k/2 \rceil$, $|X_2| = |Y_2| = \lfloor k/2 \rfloor$, and $|X_3| = |Y_3| = n - k$. Include all the edges in the following bipartite subgraphs (and no others) and color them as indicated: color $[X_1, Y_2]$, $[X_2, Y_3]$, and $[X_3, Y_1]$ blue, and color $[X_1, Y_3]$, $[X_2, Y_1]$, and $[X_3, Y_2]$ red. Since $k \leq n/3$, the order of the longest monochromatic path (and cycle) in $G$ is $2\lceil k/2 \rceil$, while the minimum degree of $G$ is $k$.
\end{example}

The following example appears in \cite{ZSW} and shows that Corollary \ref{cor1} is tight.  

\begin{example}\label{exCycle}
Let $n$ be an odd positive integer.  Let $K_{n,n}$ be an $X,Y$-bipartite graph, and partition $X$ and $Y$ into $X_1, X_2,\{x^*\}$ and $Y_1, Y_2, \{y^*\}$, respectively, where $|X_1| = |Y_1| = \lfloor n/2 \rfloor$ and $|X_2| = |Y_2| = \lfloor n/2 \rfloor$. Color $[X_1, Y_1]$ and $[X_2, Y_2]$ red, while coloring $[X_1, Y_2]$ and $[X_2, Y_1]$ blue. Color $[\{x^*\}, Y]$ blue and $[\{y^*\}, X\setminus \{x^*\}]$ red.  The order of the longest monochromatic cycle is $2\floor{n/2}$.
\end{example}

\section{Monochromatic balanced components}\label{sec:comps}

In this section we prove a result about the size of a largest monochromatic ``balanced component'' in a 2-colored bipartite graph with a given minimum degree.  The main purpose of Proposition \ref{lem:2col_comp} is to provide the first step for upcoming proof of Theorem \ref{con_match}, but since the statement is interesting on its own and is instructive to read as a warm-up, we separate it from the proof of Theorem \ref{con_match}.  Note that while Proposition \ref{lem:2col_comp} is weaker than Theorem \ref{con_match}, the bounds are still best possible by Examples \ref{exLargeDeg}, \ref{exMediumDeg}, and \ref{exSmallDeg}.

We begin with the following result, due to Liu, Morris, Prince \cite{LMP} and independently Mubayi \cite{Mub}, which strengthens an earlier result of Gy\'arf\'as \cite{Gy77}.  We provide the proof for completeness.  A \emph{double star} is a tree having at most two non-leaves.

\begin{theorem}\label{thm:gy}
In any $r$-coloring of $K_{m,n}$, there exists a monochromatic component on at least $\frac{m+n}{r}$ vertices.

In fact, if $G$ is a bipartite graph having parts of sizes $m$ and $n$, with $e(G)\geq \alpha mn$, then $G$ contains a monochromatic component on at least $\alpha(m+n)$ vertices.
\end{theorem}

\begin{proof}
The average size of a double star in $G$ is
\begin{align*}
\frac{1}{e(G)}\sum_{xy\in E(G)}(d(x)+d(y))&=\frac{1}{e(G)}\left(\sum_{x\in X}d(x)^2+\sum_{y\in Y}d(y)^2\right)\\
&\geq \frac{1}{e(G)}\left(\frac{e(G)^2}{m}+\frac{e(G)^2}{n}\right)
=e(G)\frac{m+n}{mn}
\geq \alpha(m+n),
\end{align*}
so there exists an edge $xy$ with $d(x)+d(y)\geq \alpha(m+n)$.
\end{proof}

\begin{proposition}\label{lem:2col_comp}
Let $G$ be a balanced $X,Y$-bipartite graph on $2n$ vertices.  If $\delta(G) = \delta n$, then in every 2-coloring of $G$, there exists a monochromatic component $H$ such that 
\[
|H\cap X|, |H\cap Y|\geq  
\begin{cases}
      \delta n/2,& \for 0\leq \delta \leq 2/3 \\
      (2\delta-1)n,& \for 2/3\leq \delta\leq 3/4 \\
      n/2,& \for 3/4\leq \delta\leq 1
    \end{cases}.
\]
\end{proposition}

\begin{proof}
Since $e(G)\geq \delta n^2$, one of the color classes has at least $\delta n^2/2$ edges, so by Theorem \ref{thm:gy}, there exists a monochromatic component on at least $\delta n$ vertices.

First suppose $\delta>2/3$. Let $H_1$ be the largest monochromatic, say blue, component and let $X_1=H_1\cap X$, $X_2 = X \setminus X_1$, $Y_1=H_1\cap Y$, and $Y_2 = Y \setminus Y_1$.  Let $x_i=|X_i|$ and $y_i=|Y_i|$.  We have $|H_1|\geq \delta n$, so without loss of generality, suppose $x_1\geq \delta n/2$ and $x_1\geq y_1$.  We note that if $x_1\geq y_1\geq \min\{n/2, (2\delta-1)n\}$, then we are done; so suppose $y_1<\min\{n/2, (2\delta-1)n\}$.  We have
$$
\delta_R(Y_2, X_1)\geq \delta n-x_2=x_1-(1-\delta)n\geq (3\delta/2-1)n>0
$$
and for all $u_1, u_1'\in X_1$,
$$
|N_R(u_1)\cap N_R(u_1')\cap Y_2|\geq 2(\delta n-y_1)-y_2=(2\delta -1)n-y_1>0.
$$
Which together imply that there is a red component $H_2$ covering $X_1\cup Y_2$.  Since $y_1<n/2$, we have $y_2>n/2>y_1$ and thus $|H_2|\geq x_1+y_2>x_1+y_1=|H_1|$ contradicting the maximality of $H_1$.  

Now suppose $0\leq \delta \leq 2/3$ (although we note that the following argument applies for all $\delta\geq 0$).  Let $H_1, \dots, H_k$ be the largest collection (as in $k$ is maximum) of monochromatic components having the property that $H_1, \dots, H_k$ have the same color and $|H_i\cap X|\geq \delta n/2$ for all $i\in [k]$ or $|H_i\cap Y|\geq \delta n/2$ for all $i\in [k]$.  Without loss of generality suppose $H_i$ is blue and $|H_i\cap X|\geq \delta n/2$ for all $i\in [k]$.  Let $X_i:=H_i\cap X$ and note that $|X_1\cup \dots\cup X_k|\geq k\delta n/2$.  Also note that for all $x\in X_1\cup \dots\cup X_k$, we would be done unless $d_B(x)<\delta n/2$ and consequently $d_R(x)>\delta n/2$.  So each vertex from $X_1\cup \dots \cup X_k$ is in a red component having more than $\delta n/2$ vertices in $Y$.  By the maximality of $H_1, \dots, H_k$, there are at most $k$ such red components and thus there are at least $|X_1\cup \dots\cup X_k|/k\geq \delta n/2$ vertices from $X_1\cup \dots\cup X_k$ which belong to the same red component and we are done.  
\end{proof}

\section{Monochromatic connected matchings}\label{sec:matching}

A \emph{connected matching} $M$ in a graph $G$ is a matching having the property that every edge from $M$ lies inside the same connected component in $G$.

In this section, we prove two results.  Theorem \ref{con_match} is a ``pure'' result which provides a bound on the size of a monochromatic connected matching in a 2-colored bipartite graph with a given minimum degree; this result is again best possible by Examples \ref{exLargeDeg}, \ref{exMediumDeg}, and \ref{exSmallDeg}.  We will use Theorem \ref{con_match} to prove Theorem \ref{thm2};  however, to prove Theorem \ref{thm1}, we must prove a stability version of Theorem \ref{con_match} and this is done in Theorem \ref{stable_match}.

\begin{theorem}\label{con_match}
Let $G$ be a balanced bipartite graph on $2n$ vertices.  If $\delta(G)= \delta n$, then in every 2-coloring of the edges of $G$ there exists a monochromatic connected matching of size at least 
\[
\begin{cases}
      \delta n/2,& \for 0\leq \delta \leq 2/3 \\
      (2\delta-1)n,& \for 2/3\leq \delta\leq 3/4 \\
      n/2,& \for 3/4\leq \delta\leq 1
\end{cases}.
\]
\end{theorem}

\begin{proof}
First suppose $\delta>2/3$.  Let $H_1$ be the largest monochromatic, say blue, component satisfying $|H_1\cap X|, |H_1\cap Y|\geq \min\{n/2, (2\delta-1)n\}=:m$, which exists by Proposition \ref{lem:2col_comp}.  
Note that by splitting into cases depending on whether $\delta\geq 3/4$, in which case $m=n/2$, or $2/3<\delta<3/4$, in which case $m=(2\delta-1)n$, we obtain
\begin{equation}\label{md}
m\geq 2(1-\delta)n
\end{equation}

Let $X_1=H_1\cap X$, $X_2=X\setminus X_1$, $Y_1=H_1\cap Y$, $Y_2=Y\setminus Y_1$ and $x_i=|X_i|$, $y_i=|Y_i|$ for $i\in [2]$.  Let $S$ be a minimum vertex cover of $H_1$, let $S_X=S\cap X$, and $S_Y=S\cap Y$, and note that we must have $|S|<m$ or else by K\"onig's theorem we have the desired matching.  Let $X_1'=X_1\setminus S_X$, $Y_1'=Y_1\setminus S_Y$ and $x_1'=|X_1'|$, $y_1'=|Y_1'|$.  Since $|S|<m\leq |H_1\cap X|, |H_1\cap Y|$, we have $x_1', y_1'>0$ and 
\begin{equation}\label{x1'y1'}
x_1'+y_1'=|H_1|-|S|>m.
\end{equation}

For all $u, u'\in X_1'$, we have
$$
|N_R(u)\cap N_R(u')\cap (Y\setminus S_Y)|\geq 2(\delta n-|S_Y|)-(n-|S_Y|)=(2\delta-1)n-|S_Y|>0
$$
and for all $v, v'\in Y_1'$, we have 
$$
|N_R(v)\cap N_R(v')\cap (X\setminus S_X)|\geq 2(\delta n-|S_X|)-(n-|S_X|)=(2\delta-1)n-|S_X|>0
$$
which together imply that there is a red component covering $X_1'$ and a red component covering $Y_1'$.  By \eqref{x1'y1'}, we have, say, $x_1'>m/2$.  So for all $v\in Y_1'$, we have by \eqref{md}, $d_R(v, X_1')\geq \delta n-(n-x_1')=x_1'-(1-\delta)n>m/2-(1-\delta)n\geq 0$.  Thus there is a single red component $H_2$ which covers $X_1'\cup Y_1'$.

As before, let $T$ be a minimum vertex cover of $H_2$, let $T_X=T\cap X$, and $T_Y=T\cap Y$, and note that we must have $|T|<m$ or else by K\"onig's theorem we have the desired matching.  Note that by \eqref{x1'y1'}, $(X_1'\cup Y_1')\setminus T\neq \emptyset$, so without loss of generality, say $u\in X_1'\setminus T$.  Note that $u\in H_1\cap H_2$ and $u\not\in S\cup T$, so $N(u)\subseteq S_Y\cup T_Y$, which implies $|S_Y|+|T_Y|\geq \delta n$ and thus 
\begin{equation}\label{SXTX}
|S_X|+|T_X|<2m-\delta n.  
\end{equation}
Furthermore, since $|S_Y|+|T_Y|<2m\leq n$, there exists $v'\in Y\setminus (S_Y\cup T_Y)$.  Note that $N(v')\subseteq (S_X\cup T_X\cup X_2)$ and thus by \eqref{SXTX}, $$\delta n\leq |S_X\cup T_X\cup X_2|\leq |S_X|+|T_X|+x_2<2m-\delta n+x_2\leq 2m-\delta n+n-m,$$
and thus $m>(2\delta-1)n$, a contradiction.

Now suppose $0\leq \delta \leq 2/3$ (although we note that the following argument applies for all $\delta\geq 0$). Define the following four sets:
\[L_X^B = \{ v \in X : d_B(v) \geq \delta n/2 \}; \quad L_X^R = \{ v \in X : d_R(v) \geq \delta n/2 \};\]
\[L_Y^B = \{ v \in Y : d_B(v) \geq \delta n/2 \}; \quad L_Y^R = \{ v \in Y : d_R(v) \geq \delta n/2 \}.\]
Without loss of generality let $L_X^B$ be a set of maximum order (out of the four sets).  Let $H_1, \dots, H_t$ be the blue components which intersect $L_X^B$ and for each $i\in [t]$, let $L_i=L_X^B\cap H_i$, let $X_i=H_i\cap X$, and let $Y_i=H_i\cap Y$. For each $i\in [t]$, let $C_i$ be a minimum vertex cover for $H_i$ and suppose that $|C_i|<\delta n/2$ for all $i\in [t]$, otherwise we would be done by K\"onig's theorem.  If there exists $v \in L_i \setminus C_i$, then $N_B(v) \subseteq C_i\cap Y$, but $|C_i\cap Y|\leq |C_i| < \delta n/2 \leq d_B(v)$. Thus $L_i \subseteq C_i$, so $k_i:=|C_i\cap X| \geq |L_i|$.  Note that since $|Y_i|\geq \delta n/2$ and $|C_i|<\delta n/2$, we have
\[
|Y_i \setminus C_i| \geq |Y_i| - (|C_i| - k_i) > k_i\geq  |L_i|.
\] 
Also for all $i\in [t]$ and all $v\in Y_i\setminus C_i$ we have $d_B(v) \leq k_i \leq |C_i| < \delta n/2$, which implies $d_R(v) > \delta n/2$ and thus $Y_i\setminus C_i \subseteq L_Y^R$.  So 
\[
|L_Y^R|\geq \sum_{i\in[t]}|Y_i\setminus C_i|> \sum_{i\in [t]}|L_i|=|L_X^B|
\]
contradicting the choice of $L_X^B$.
\end{proof}

\begin{theorem}\label{stable_match}
For all $\eta>0$ there exists $n_0$ such that if $G$ is a balanced bipartite graph on $2n\geq 2n_0$ vertices with $\delta(G)> (3/4+\eta)n$, then in every 2-coloring of the edges of $G$ there exists a monochromatic connected matching of size at least $(1/2+\eta)n$ or the coloring of $G$ is $2\eta$-extremal.
\end{theorem}

\begin{proof}
Let $H_1$ be a largest monochromatic, say blue, component satisfying $|H_1\cap X|, |H_1\cap Y|\geq n/2$ (such an $H_1$ exists by Proposition \ref{lem:2col_comp}).  Let $X_1=H_1\cap X$, $X_2=X\setminus X_1$, $Y_1=H_1\cap Y$, $Y_2=Y\setminus Y_1$ and $x_i=|X_i|$, $y_i=|Y_i|$ for $i\in [2]$.  Note that all edges in $[X_1, Y_2]$ and $[Y_1, X_2]$ are red and for all $u,u'\in X_2$,
\begin{equation}\label{x2y1}
|(N_R(u)\cap N_R(u'))\cap Y_1|\geq 2((3/4+\eta)n-y_2)-y_1=(1/2+2\eta)n-y_2\geq 2\eta n>0
\end{equation}
and for all $v,v'\in Y_2$,
\begin{equation}\label{y2x1}
|(N_R(v)\cap N_R(v'))\cap X_1|\geq 2((3/4+\eta)n-x_2)-x_1=(1/2+2\eta)n-x_2\geq 2\eta n>0.
\end{equation}
Furthermore, 
\begin{equation}\label{12}
\delta_R(X_1,Y_2)\geq (3/4+\eta)n-y_1 ~\text{ and }~ \delta_R(Y_1,X_2)\geq (3/4+\eta)n-x_1.
\end{equation}

First suppose $y_1<(1/2+\eta)n$ or $x_1<(1/2+\eta)n$.  If $y_1<(1/2+\eta)n$, then by \eqref{y2x1} and \eqref{12}, there is a red component covering $[X_1, Y_2]$.  By the maximality of $H_1$, there are fewer than $\eta n$ vertices in $Y_1$ which have a red neighbor in $X_1$, so we have an $\eta$-extremal coloring as witnessed by $X_1, Y_1, Y_2$.  A similar calculation shows that if $x_1<(1/2+\eta)n$, then we have an $\eta$-extremal coloring as witnessed by $Y_1, X_1, X_2$.

So suppose $x_1\geq (1/2+\eta)n$ and $y_1\geq (1/2+\eta)n$.
Let $S$ be a minimum vertex cover of $H_1$, let $S_X=S\cap X$, and $S_Y=S\cap Y$ and note that we must have $|S|<(1/2+\eta)n$ or else we have the desired matching in $H_1$ by K\"onig's theorem.  Let $X_1'=X_1\setminus S_X$ and $Y_1'=Y_1\setminus S_Y$ and note that all edges in $[X_1', Y_1'\cup Y_2]$ and $[Y_1', X_1'\cup X_2]$ are red.  Set $x_1'=|X_1'|$ and $y_1'=|Y_1'|$.  By the supposition on $x_1$ and $x_2$ and by the bound on $|S|$ we have 
\begin{equation}\label{X1Y1}
x_1'+y_1'=x_1+y_1-|S|>(1/2+\eta)n ~\text{ and }~ x_1', y_1'>0
\end{equation}

For all $u,u'\in X_1'$, we have 
\begin{equation*}
|(N_R(u)\cap N_R(u'))\cap (Y\setminus S_Y)|\geq 2((3/4+\eta)n-|S_Y|)-(n-|S_Y|)=(1/2+2\eta)n-|S_Y|>0
\end{equation*}
and for all $v,v'\in Y_2'$, we have
\begin{equation*}
|(N_R(v)\cap N_R(v'))\cap (X\setminus S_X)|\geq 2((3/4+\eta)n-|S_X|)-(n-|S_X|)=(1/2+2\eta)n-|S_Y|>0
\end{equation*}
and thus there is a red component covering $X_1'$ and a red component covering $Y_1'$.  Also, by \eqref{X1Y1} we have, say $x_1'\geq \frac{1}{2}(|H_1|-|S|)>n/4$.  Thus every vertex in $Y_1'$ has a red neighbor in $X_1'$ and thus there is a single red component $H_2$ covering $X_1'\cup Y_1'$.

Let $T$ be a minimum vertex cover of $H_2$, let $T_X=T\cap X$, and $T_Y=T\cap Y$, and again, by K\"onig's theorem, we may assume $|T|<(1/2+\eta)n$.  By \eqref{X1Y1}, $|T|<x_1'+y_1'$ and thus there exists some vertex $u\in (X_1'\cup Y_1')\setminus (S\cup T)$, say $u\in X_1'\setminus (S_X\cup T_X)$.  Then we must have $N(u)\subseteq S_Y\cup T_Y$ which implies $|S_Y|+|T_Y|\geq (3/4+\eta)n$, which in turn  implies 
\begin{equation}\label{sXtX}
|S_X|+|T_X|=|S|+|T|-(|S_Y|+|T_Y|)\leq 2(1/2+\eta)n-(3/4+\eta)n=(1/4+\eta)n.  
\end{equation}

If there exists $v\in Y\setminus (S_Y\cup T_Y)$, then we have $N(v)\subseteq S_X\cup T_X\cup X_2$, but by \eqref{sXtX}, we then have 
$$(3/4+\eta)n\leq |S_X|+|T_X|+x_2\leq (1/4+\eta)n+(1/2-\eta)n= 3n/4,$$ a contradiction.  So we must have $Y\subseteq S_Y\cup T_Y$ and thus $|S_X|+|T_X|<2\eta n$.  Set $T_Y'=T_Y\setminus S_Y$.  Note that the only red edges from $X_1'$ to $S_Y$ must be incident with $T_X$ and the only blue edges from $X_1'$ to $T_Y'$ must be incident with $S_X$.
So we have $x_1'\geq x_1-|S_X|-|T_X|\geq (1/2-\eta)n$, $|S_Y|, |T_Y'|\geq (1/2-\eta)n$, 
$$e_R(S_Y, X_1')\leq |S_Y||T_X|\leq 2\eta n^2,$$ and 
$$e_B(T_Y', X_1')\leq |T_Y'\cap Y_1||S_X|\leq 2\eta n^2,$$
thus the coloring is $2\eta$-extremal as witnessed by $X_1', S_Y, T_Y'$.  

\end{proof}

\section{Regularity: from connected matchings to cycles}\label{sec:reg}

In this section, we introduce the now standard machinery which allows us to reduce the problem of find a long monochromatic cycle to the problem of finding a large monochromatic connected matching.

Given a graph $G$ and disjoint sets $X,Y\subseteq V(G)$, define $d(X,Y)=\frac{e(X,Y)}{|X||Y|}$.  Given $\ep\geq 0$, say that a bipartite graph with parts $X,Y$ is $\ep$-\emph{regular} if $|d(X',Y')-d(X,Y)|\leq \ep$ for all $X'\subseteq X$ and $Y'\subseteq Y$ with $|X'|> \ep |X|$ and $|Y'|>\ep |Y|$.    

Below is the standard degree form for the $2$-colored regularity lemma (see \cite{KS}) in which we begin with an initial bipartition of the vertex set.  We call $\{E_1, E_2\}$ a $2$-multicoloring of $G$ if $E_1\cup E_2=E(G)$ (i.e. we allow for $E_1\cap E_2\neq \emptyset$).

\begin{lemma}[2-colored regularity lemma -- bipartite degree form]\label{2colordegreeform}
For all $0< \ep\leq 1$ and positive integers $m$, there exists an $M = M(\ep,
m)$ such that for all 2-colored balanced $X,Y$-bipartite graphs $G$ on $2n\geq M$ vertices and all
$d\in [0,1]$, there exists an integer $k$, a partition $\{U_0,
U_1,\dots, U_k\}$ of $X$ and a partition $\{V_0,
V_1,\dots, V_k\}$ of $Y$, and a subgraph $G' \subseteq G$ with the following properties:
\begin{enumerate}
\item $|U_0|=|V_0| \le \ep n$
\item $m \le k \le M$ and $|V_1|=\dots=|V_k|=|U_1|=\dots=|U_k|$,
\item $d_{G'}(v)>d_G(v)-(2d+\ep)n$ for all $v\in V(G)$,
\item for all $1 \le i \le j \le k$, the pair $(U_i,V_j)$ is $\ep$-regular in $G_R'$ with a density either 0 or greater than $d$ and $\ep$-regular in $G_B'$ with a density either 0 or greater than $d$, where $E(G') = E(G_R') \cup E(G_B')$ is the induced 2-coloring of $G'$.
\end{enumerate}
\end{lemma}

\begin{definition}[$(\ep,d)$-reduced graph]\label{def:reduced}
Given an $X,Y$-bipartite graph $G$ and partitions $\{U_0, U_1,\dots, U_k\}$ of $X$ and $\{V_0,
V_1,\dots, V_k\}$ of $Y$ satisfying conditions (i)-(v) of Lemma \ref{2colordegreeform}, we define the $(\ep, d)$-reduced graph of $G$ to be the bipartite graph $\Gamma$ on vertex set $\{U_1, \dots, U_k\} \cup \{V_1, \dots, V_k\}$ such that $U_iV_j$ is an edge of $\Gamma$ if $G'[U_i, V_j]$ has density at least $2d$.  For each $U_iV_j\in E(\Gamma)$, we assign red if $G'_R[U_i, V_j]$ has density at least $d$, and blue if $G'_B[U_i, V_j]$ has density at least $d$. 
\end{definition}

The following is a well known consequence of the regularity lemma (see Proposition 42 in \cite{KO}).

\begin{lemma}\label{reduceddegree}
Let $0<2\ep\leq d\leq c/2$ and let $G$ be a graph on $n$ vertices with $\delta(G)\geq cn$. If $\Gamma$ is a $(\ep, d)$-reduced graph of $G$ obtained by applying Lemma \ref{2colordegreeform}, then $\delta(\Gamma)\geq (c-3d)k$.
\end{lemma}

Finally we state the lemma which allows us to turn the connected matching in the reduced graph into the cycle in the original graph. Some variant of this lemma, first introduced by \L uczak \cite{Lu}, has been utilized by many authors, in particular \cite{GS}, \cite{BLSSW}, and \cite{W}.  See Lemma 2.2 in \cite{BLSSW} for the variant of \L uczak's lemma which is used to build the nearly spanning paths in each pair (in place of the much stronger blow-up lemma).  

\begin{lemma}\label{match_cycle}
Let $0<\ep\ll d$ and let $\Gamma$ be an $(\ep, d)$-reduced graph of a $2$-colored graph $G$.  Assume that there is a monochromatic connected matching $M$ saturating at least $c|V(\Gamma)|$ vertices of $\Gamma$, for some positive constant $c$.  If $U\subseteq V(G)$ is the set of vertices spanned by the clusters in $M$, then there is a monochromatic cycle in $G$ covering at least $c(1-6\sqrt{\ep})n$ vertices of $U$.  
\end{lemma}

\section{Monochromatic circumference}\label{sec:circ}

For a fixed positive integer $r$, the \emph{monochromatic circumference} of a graph $G$ is largest value of $t$ such that in every $r$-coloring of the edges of $G$, there exists a monochromatic cycle of length at least $t$.  We now prove Theorem \ref{thm2} which provides an asymptotically sharp bound on the monochromatic circumference of a balanced bipartite graph with a given minimum degree in the case $r=2$. 

\begin{proof}[Proof of Theorem \ref{thm2}]
As in Section 1, let
\[
f(\delta)=\begin{cases}
      \delta /2,& \for 0\leq \delta \leq 2/3, \\
      2\delta-1,& \for 2/3\leq \delta\leq 3/4, \\
      1/2,& \for 3/4\leq \delta\leq 1.
\end{cases}
\]
Let $0 < \ep \ll d \ll \eta$. Apply Lemma \ref{2colordegreeform} to $G$ to get a balanced bipartite $(\ep,d)$-reduced graph $\Gamma$ on $2k$ vertices with minimum degree at least $(\delta-3d)k$ and then apply Lemma \ref{con_match} to $\Gamma$ to get a connected matching of size at least $f(\delta-3d) k\geq (f(\delta)-6d)k$.  Now apply Lemma \ref{match_cycle} to get a cycle of length at least $(1-6\sqrt{\ep})(f(\delta)-6d)k\geq (f(\delta)-\eta) n$.
\end{proof}

\section{Stability}\label{sec:stability}

In this section we prove a lemma which shows that if we have an extremal coloring of the reduced graph, then we have an extremal coloring of the original graph (with a slightly weaker parameter).  We use this together with Theorem \ref{stable_match} to prove Theorem \ref{thm1}. 

\begin{lemma}\label{reducedextremal}
Let $0< 2\ep\leq d$ and $4\ep+d\leq \eta \leq 1/4$ and let $n_0$ be sufficiently large.  Let $G$ be a balanced bipartite graph on $2n\geq 2n_0$ vertices with 2-edge-coloring $c$, and let $\Gamma$ be a 2-colored $(\ep, d)$-reduced graph of $G$, with edge coloring $c'$, after an application of Lemma \ref{2colordegreeform}.  If $c'$ is an $\eta$-extremal coloring of $\Gamma$, then $c$ is an $2\eta$-extremal coloring of $G$.
\end{lemma}

\begin{proof}
Let $\cU=\{U_1, \dots, U_k\}$, $\cV=\{V_1, \dots, V_k\}$ be the bipartition of $\Gamma$ and let $m:=|U_1|=\dots=|U_k|=|V_1|=\dots=|V_k|$.  Suppose we have $\cU'\subseteq \cU$ and a partition $\cV_1, \cV_2$ of $\cV$ such that  such that $|\cU'|, |\cV_1|, |\cV_2|\geq (1/2-\eta)k$, $e_R(\cU', \cV_1)\leq \eta |\cU'||\cV_1|$, and $e_B(\cU',\cV_2)\leq \eta |\cU'||\cV_2|$.  Let $X'=\bigcup_{U\in \cU'}U$, $Y_1=V_0\cup \bigcup_{V\in \cV_1} V$, and $Y_2=\bigcup_{V\in \cV_2} V$.  Since $|V_0|\leq \ep n$, we have 
$$e_R(X', Y_1)\leq \ep n|X'|+\eta |\cU'||\cV_1|m^2+d|\cU'||\cV_1|m^2\leq\left(\frac{\ep}{1/2-\eta}+\eta+d\right) |X'||Y_1|\leq 2\eta|X'||Y_1|$$ and 
$$e_B(X', Y_2)\leq \eta |\cU'||\cV_2|m^2+d|\cU'||\cV_2|m^2=(\eta+d)|X'||Y_2|\leq 2\eta|X'||Y_2|.$$  Furthermore, $|X'|\geq (1-\ep)(n/k)|\cU'|\geq (1-\ep)(1/2-\eta)n\geq (1/2-2\eta)n$ and $|Y_i|\geq (1-\ep)(n/k)|\cV_i|\geq (1-\ep)(1/2-\eta)n\geq (1/2-2\eta)n$ for $i\in [2]$.  So $G$ is $2\eta$-extremal.  
\end{proof}

Now we prove the main result.  We delay the proof of Proposition \ref{extremalCase} (the extremal case) to the final section.

\begin{proof}[Proof of Theorem \ref{thm1}]
Let $0\leq 64\sqrt{\eta}<\gamma\leq 1/4$, let $0 < \frac{1}{n_0}\ll \ep \ll d \ll \gamma \leq \frac{1}{4}$, and let $G$ be a graph on $n\geq n_0$ vertices with $\delta(G)\geq (3/4+\gamma)n$.  Apply Lemma \ref{2colordegreeform} to $G$ to get a balanced bipartite $(\ep,d)$-reduced graph $\Gamma$ on $2k$ vertices with minimum degree at least $(3/4+\gamma/2)k$ and then apply Theorem \ref{stable_match} to $\Gamma$.  If $\Gamma$ has a monochromatic connected matching of size at least $(1/2+2\eta)k$, then apply Lemma \ref{match_cycle} to get a cycle of length at least $(1+\eta)n$.  Otherwise, $\Gamma$ is $4\eta$-extremal, so by Lemma \ref{reducedextremal}, $G$ is $8\eta$-extremal and since $16\sqrt{8\eta}\leq 64\sqrt{\eta}\leq  \gamma$ we may apply Proposition \ref{extremalCase} to $G$ to finish the proof.
\end{proof}

\section{Extremal Case}\label{sec:extremal}

In this section we complete the proof of Theorem \ref{thm1} by showing that if $G$ has an extremal coloring, then $G$ has a monochromatic path of order $2\ceiling{n/2}$ and a monochromatic cycle of length at least $2\floor{n/2}$.  

We utilize the following two theorems to find long monochromatic paths and cycles.

\begin{theorem}[Erd\H{o}s, Gallai \cite{EG}]\label{erdos_lemma}
Let $G$ be a graph on $n$ vertices. If $e(G) > \frac{k(n-1)}{2}$, then $G$ contains a cycle of length at least $k+1$.
\end{theorem}

We say that a balanced $X,Y$-bipartite graph is Hamiltonian bi-connected if for all $x\in X$ and $y\in Y$, there exists a Hamiltonian path having $x$ and $y$ as endpoints.  The following is a Chv\'atal-type theorem for a bipartite graph to be Hamiltonian bi-connected.

\begin{theorem}[see Berge {\cite[Chapter 10, Theorem 14]{Ber}}]\label{berge_lemma}
Let $G=(U,V,E)$ be a bipartite graph on $2m\geq 4$ vertices with vertices in $U=\{u_1, \dots, u_m\}$ and $V=\{v_1, \dots, v_m\}$ such that $d(u_1) \leq \cdots \leq d(u_m)$ and $d(v_1) \leq \cdots \leq d(v_m)$.  If for the smallest two indices $j$ and $k$ such that $d(u_j)\leq j+1$ and $d(v_k)\leq k+1$, we have $$d(u_j)+d(v_k)\geq m+2,$$
then $G$ is Hamiltonian bi-connected.
\end{theorem}

Before tackling the main extremal case, we first prove two useful lemmas.

\begin{lemma}\label{longPathProp}
Let $0 \leq 8\sqrt{\theta} < \gamma \leq \frac{1}{4}$ be real numbers and let $n$ be an integer such that $n\geq 3/\gamma$. Let $G$ be a 2-colored balanced $X,Y$-bipartite graph on $2n$ vertices with $\delta(G) \geq (\frac{3}{4} + \gamma) n$.  Given $X' \subseteq X$ and $Y' \subseteq Y$ with $ |X'| \geq (\frac{1}{2} - \gamma/8)n$ and $ |Y'| \geq (\frac{1}{2} - \gamma/8)n $, and $e_R(X', Y') \leq \theta n^2$, define 
\[X_S = \{x \in X' : d_B(x, Y') \leq |Y'|-n/2+\gamma n/2\}\] and 
\[Y_S = \{y \in Y' : d_B(y, X') \leq |X'|-n/2 + \gamma n/2\}.\] 
Then $|X_S|, |Y_S| \leq 4\theta n$ and for all $X^* \subseteq X' \setminus X_S$ and $Y^* \subseteq Y' \setminus Y_S$ such that $|X^*| = |Y^*| \geq  (\frac{1}{2} - \gamma/4) n$, $[X^*, Y^*]_B$ is Hamiltonian bi-connected.
\end{lemma}

\begin{proof}
Note that
\begin{align*}
\delta_R(X_S, Y') &\geq \delta(G)- (n-|Y'|) - \Delta_B(X_S, Y')\\
&\geq (3/4+\gamma)n-(n-|Y'|)-(|Y'|-n/2+\gamma n/2)=n/4,
\end{align*}
so
\[ 
\theta n^2 \geq e_R(X', Y') \geq e_R(X_S, Y') \geq |X_S|n/4 ,
\]
which implies $|X_S| \leq 4\theta n$. Similarly $|Y_S| \leq 4 \theta n$.

Set $s:=|X'|-(1/2-\gamma/8)n$ and $t:=|Y'|-(1/2-\gamma/8)n$.  Set $X_L = X' \setminus X_S$ and $Y_L = Y' \setminus Y_S$. Note that
\[ |X_L| \geq |X'|-4\theta n\geq (\frac{1}{2} - \gamma/4) n ; \quad |Y_L| \geq |Y'|-4\theta n\geq (\frac{1}{2} - \gamma/4) n,\]
and
\begin{equation}\label{XLYL}
\delta_B(X_L, Y_L) \geq |Y'|-n/2+\gamma n/2-|Y_S|\geq t+\gamma n/4 ; \quad \delta_B(Y_L, X_L) \geq s + \gamma n/4.
\end{equation}


Now take any $X^* \subseteq X_L$ and $Y^* \subseteq Y_L$ with $m:= |X^*| = |Y^*| \geq (\frac{1}{2} - \gamma/4) n$. 
Define
\[ X_S' = \{x \in X^* : d_B(x, Y^*) \leq m-(1/4-\gamma/2)n\}, \text{ and}\]
\[ Y_S' = \{y \in Y^* : d_B(y, X^*) \leq m-(1/4-\gamma/2)n\} .\]
Then
\begin{align*}
\delta_R(X_S', Y^*) &\geq \delta(G) - (n-|Y^*|) - \Delta_B(X_S', Y^*)\\
&\geq (3/4+\gamma)n-(n-m)-(m-(1/4-\gamma/2)n)=\gamma n/2 ,
\end{align*}
so
\[ \theta n^2 \geq e_R(X', Y')\geq  e_R(X_L, Y_L) \geq e_R(X_S', Y^*) \geq |X_S'|\gamma n/2 ,\]
which implies $|X_S'| \leq \frac{2\theta}{\gamma}n$. Similarly $|Y_S'| \leq \frac{2\theta}{\gamma} n$.

Note that from \eqref{XLYL} and $|Y_L|-|Y^*| \leq t+\gamma n/8$, we have 
\begin{equation}\label{X*Y*}
\delta_B(X^*, Y^*) \geq \delta_B(X_L, Y_L) - (t+\gamma n/8) \geq \gamma n/8,
\end{equation}
and similarly 
\begin{equation}\label{Y*X*}
\delta_B(Y^*, X^*) \geq \gamma n/8.
\end{equation}

Now enumerate $X^*$ and $Y^*$ as $x_1,\dots, x_m$ and $y_1,\dots, y_m$, respectively, in increasing order of degree in $[X^*, Y^*]_B$. Let $i$ and $j$ be the smallest indices such that $d_B(x_i, Y^*) \leq i+1$ and $d_B(y_j, X^*) \leq j+1$. Either $X_S'=\emptyset$ or by \eqref{X*Y*}
\[
\delta_B(X_S', Y^*) \geq \gamma n/8 \geq \frac{6\theta}{\gamma} n \geq 3|X_S'|\geq |X_S'|+2;
\]
either way we have $i > |X_S'|$, so $x_i \in X^* \setminus X_S'$.  Similarly, by \eqref{Y*X*}, we have $y_j \in Y^* \setminus Y_S'$. Then 
\[
d_B(x_i, Y^*) + d_B(y_j, X^*) > 2(m-(1/4-\gamma/2)n)=m + (m-(1/2-\gamma)n) \geq m + 2, 
\]
so by Theorem \ref{berge_lemma}, $[X^*, Y^*]_B$ is Hamiltonian bi-connected.
\end{proof}

\begin{lemma}\label{BigPartProp}
Let $0 \leq 8\sqrt{\theta} < \gamma \leq \frac{1}{4}$ be real numbers and let $n$ be an integer such that $n\geq 3/\gamma$. Let $G$ be a 2-colored balanced $X,Y$-bipartite graph on $2n$ vertices with $\delta(G) \geq (\frac{3}{4} + \gamma)n$. If there exists $X' \subseteq X$, $Y' \subseteq Y$ such that $|X'| \geq 3n/4$ and $(\frac{1}{2} + \theta) n \geq |Y'| \geq \frac{1}{2} n$, $e_R(X', Y') \leq \theta n^2$, and $\delta_B(Y', X') \geq \gamma n$, then $G$ contains a blue cycle on $2\ceiling{n/2}$ vertices. 
\end{lemma}

\begin{proof}
Define $X_S = \{x \in X' : d_B(x, Y') \leq (1/4+3\gamma/4)n \}$. Note that
\begin{align*}
\theta n^2\geq e_R(X', Y')\geq |X_S|\delta_R(X_S, Y')&\geq |X_S|((3/4+\gamma)n-(1/4+3\gamma /4)n-(n-|Y'|))\\
&\geq |X_S|\gamma n/4
\end{align*}
and thus $|X_S| \leq  \frac{4\theta}{\gamma} n \leq \sqrt{\theta} n$.
Set $X_L = X' \setminus X_S$. Define $Y_S = \{y \in Y' : d_B(y, X_L) \leq |X_L| - n/2 + 3\gamma n/4 \}$. Note that
\begin{align*}
\theta n^2\geq e_R(Y', X_L)&\geq |Y_S| \delta_R(Y_S, X_L)\\
&\geq |Y_S|((3/4+\gamma)n-(n-|X_L|)-(|X_L|-n/2+3\gamma/4) \geq  |Y_S| n/4
\end{align*}
and thus $t:=|Y_S|\leq 4\theta n$.
Let $Y_L = Y' \setminus Y_S$ and enumerate the vertices of $Y_S$ as $v_1, \dots, v_t$. Note that for all $v \in Y_S$,
\begin{equation}\label{XL}
|N_B(v) \cap X_L| \geq \gamma n - |X_S|\geq \gamma n- \sqrt{\theta} n \geq 8\theta n\geq 2t
\end{equation}
and for all $x, x' \in X_L$,
\begin{align}
|N_B(x) \cap N_B(x') \cap Y_L| &\geq 2((1/4 + \gamma/2)n-|Y_S|) - |Y_L|\notag\\
&\geq (1/2+\gamma-8\theta)n-(1/2+\theta)n >4\theta\geq t.\label{YL}
\end{align}

By \eqref{XL}, for all $i\in [t]$ we can greedily find $x_i, x_i'$ such that $x_iv_i$ and $v_ix_i'$ are blue edges.  Now by \eqref{YL} we may then greedily find $v_i'\in Y_L$ for $1 \leq i \leq t$ such that $x_i'v_i'$ and $v_i'x_{i+1}$ are blue edges. Then $x_1 v_1 x_1' v_1' \dots x_t v_t x_t' v_t'$ is a blue path $P$ covering $Y_S$. Applying Lemma \ref{longPathProp} with $X':=X_L \setminus (V(P) \setminus \{x_1\})$ and $Y':=Y' \setminus (V(P) \setminus \{v_t'\})$ (note that in the application $X_S=\emptyset=Y_S$), we get $X^* \subseteq X_L$ and $Y^* \subseteq Y'$ such that $|X^*| = |Y^*| \geq \frac{n}{2} - 2t$, $X^* \cap V(P) = \{x_1\}$, $Y^* \cap V(P) = \{v_t'\}$, and $[X^*, Y^*]_B$ is Hamiltonian connected.  We can thus get a blue path $P'$ with endpoints $x_1$ and $v_t'$ such that $V(P') \cap Y = Y^*$. Then $P$ joined with $P'$ is a blue cycle on at least $2\ceiling{n/2}$ vertices.
\end{proof}

Now we prove the main result of this section.

\begin{proposition}\label{extremalCase}
Let $\eta$ and $\gamma$ be real numbers with $0 \leq 16 \sqrt{\eta} < \gamma \leq \frac{1}{4}$ and let $n$ be an integer with $n \geq 3/\gamma$. If a balanced $X,Y$-bipartite graph $G$ on $2n$ vertices with $\delta(G)\geq (3/4+\gamma)n$ is $\eta$-extremal, then $G$ has a monochromatic path of order at least $2\ceiling{n/2}$ and a monochromatic cycle of length at least $2\floor{n/2}$.
\end{proposition}

\begin{proof}
Our goal throughout the proof will be to find a monochromatic cycle of length at least $2\ceiling{n/2}$ which will satisfy both conclusions. We will be able to do this in all but one case (which is necessary because of Example \ref{exCycle}).

Since $G$ is $\eta$-extremal, let $X_1, X_2$ and $Y_1, Y_2$ partition $X$ and $Y$, respectively, with $|X_1|, |Y_1|, |Y_2| \geq (1/2 - \eta) n$ and $e_R(X_1, Y_1), e_B(X_1, Y_2) \leq \eta n^2$. Let $s := |X_1| - \frac{n}{2}$ (note that $s$ could be negative).

Define the following:
\[ Y_1^S = \{v \in Y_1 : d_B(v, X_1) \leq s + \gamma n\}; \quad Y_2^S = \{v \in Y_2 : d_R(v, X_1) \leq s + \gamma n\}; \]
\[ Y_1' = (Y_1 \setminus Y_1^S) \cup Y_2^S; \quad Y_2' = (Y_2 \setminus Y_2^S) \cup Y_1^S = Y \setminus Y_1' .\]
Then
\begin{align*}
\eta n^2 \geq e_R(X_1, Y_1) \geq e_R(X_1, Y_1^S)&\geq |Y_1^S|\delta_R(Y_1^S, X_1)\\
&\geq |Y_1^S|((3/4+\gamma)n-(s+\gamma n)-(n/2-s)) \geq |Y_1^S|n/4 ,
\end{align*}
so $|Y_1^S| \leq 4\eta n$. Similarly $|Y_2^S| \leq 4\eta n$.  So we have
\[ |Y_1'| \geq |Y_1| - |Y_1^S| \geq \left(\frac{1}{2} - 5\eta \right)n ; \quad |Y_2'| \geq |Y_2| - |Y_2^S| \geq \left(\frac{1}{2} - 5\eta \right)n; \]
\[ e_R(X_1, Y_1') \leq e_R(X_1, Y_1) + e_R(X_1, Y_2^S) \leq 5\eta n^2 ; \quad e_B(X_1, Y_2') \leq 5\eta n^2.\]

Without loss of generality, suppose $|Y_1'| \geq |Y_2'|$ so we have
\[|Y_1'|\geq \ceiling{n/2}.\]
Define the following:
\[ X_1^S = \{x \in X_1 : d_B(x, Y_1') \leq \gamma n\}; \quad X_2^S = \{x \in X_2 : d_R(x, Y_1') \leq \gamma n\}; \]
\[ X_1' = (X_1 \setminus X_1^S) \cup A_2^S; \quad X_2' = (X_2 \setminus A_2^S) \cup X_1^S ,\]
where $A_2^S=\emptyset$ if $|X_1|-|X_1^S|\geq n/2$ and otherwise $A_2^S \subseteq X_2^S$ is largest possible with size at most $\frac{n}{2} - |X_1 \setminus X_1^S|$ (note that, as opposed to $|X_1^S|$, it is possible for $|X_2^S|$ to be large, so we only want to move as many vertices from $X_2^S$ as necessary).
Then
\begin{align*} 
5\eta n^2 \geq e_R(X_1, Y_1') \geq e_R(X_1^S, Y_1')\geq |X_1^S|\delta_R(X_1^S, Y_1')
\geq |X_1^S|(3n/4-|Y_2'|) \geq |X_1^S|n/4 ,
\end{align*}
so $|X_1^S| \leq 20 \eta n$, which implies that 
\begin{equation}\label{X1'lower}
|X'_1| \geq |X_1| - |X_1^S| \geq (\frac{1}{2} - 20\eta )n + s
\end{equation}
 and $|A_2^S| \leq \min\{|X_2^S|, \max\{0, |X_1^S|-s\}\}\leq 21 \eta n$. So we have
\[
\delta_B(Y_1', X'_1) \geq s+\gamma n-|X_1^S|\geq s + 3\gamma n/4; \quad \delta_B(X'_1, Y_1') > \gamma n
\]
and
\begin{equation}\label{X1'toY1'}
e_R(X'_1, Y_1') \leq e_R(X_1, Y_1') + e_R(A_2^S, Y_1') \leq 5\eta n^2 + |A_2^S|\gamma n \leq 9\eta n^2. 
\end{equation}
 Similar to \eqref{X1'toY1'}, we also have that 
\begin{equation}\label{Y2'toX1'}
e_B(Y_2', X_1') \leq 9 \eta n^2.
\end{equation}

If $|X'_1| \geq \frac{3}{4}n$ and thus we have $\delta_B(Y_1', X'_1) \geq s+3\gamma n/4 \geq n/4$.  By \eqref{X1'toY1'}, we may use Lemma \ref{BigPartProp} with $X':=X_1'$ $Y':=Y_1'$, and $\theta:=9\eta$ to get a cycle of length at least $2\ceiling{n/2}$.

If $\frac{3}{4}n > |X'_1| \geq \frac{n}{2}$, then $s \leq \frac{n}{4} + 20 \eta n$ and we apply Lemma \ref{longPathProp} with $X' := X'_1$, $Y' := Y_1'$, $\theta := 9 \eta$ (note that in the application $X_S=\emptyset=Y_S$) to get a cycle of length at least $2\ceiling{n/2}$.

Otherwise we have $n/2+s-20\eta n\leq |X_1'|<n/2$ and thus $|X'_2| > \frac{n}{2}$.  
So if $e_B(X_2', Y_1') \geq 12\eta n^2$, then we may use Theorem \ref{erdos_lemma} to get a path $P \subseteq [X_2', Y_1']_B$ such that $k := |V(P) \cap X_2'| + 1 = |V(P) \cap Y_1'| = \ceiling{24\eta n}$ with endpoints $y, y' \in Y_1'$. 
Extend $P$ to $P'$ using a  blue edge $y'x \in [Y_1', X_1']$. Now using Lemma \ref{longPathProp} on $[X_1', Y_1']_B$ with $X':=X_1'$, $Y':=Y_1'$, and $\theta := 9 \eta$ (note that in the application $X_S=\emptyset=Y_S$), we get $X^* \subseteq X_1'$ and $Y^* \subseteq Y_1'$ such that $[X^*,Y^*]_B$ is Hamiltonian bi-connected, $x \in X^*$, $Y^* \cap V(P') = \{y\}$, and $|X^*| = |Y^*| = |Y_1'| - k + 1$.  Take a Hamiltonian path in $[X^*, Y^*]_B$ with endpoints $x$ and $y$, and adjoin it to $P'$ to form a cycle of length at least $2 \lceil n/2 \rceil$.

So suppose we are in the case that $e_B(X_2', Y_1') < 12 \eta n^2$. Use Lemma \ref{longPathProp} with $X':=X_2'$, $Y':=Y_1'$, and $\theta = 12 \eta$ to get $X_2^* \subseteq X_2'$ and $Y_1^* \subseteq Y_1'$ such that $[X_2^*,Y_1^*]_R$ is Hamiltonian bi-connected and $|X_2^*| = |Y_1^*| \geq \frac{n}{2} - \frac{\gamma}{4}n$. By \eqref{Y2'toX1'}, we may use Lemma \ref{longPathProp} again with $X':=X_1'$, $Y':=Y_2'$, and $\theta = 9\eta$ to get $X_1^* \subseteq X_1'$ and $Y_2^* \subseteq Y_2'$ such that $[X_1^*, Y_2^*]_R$ is Hamiltonian bi-connected\footnote{Note that we know more than just that $[X_1^*,Y_2^*]_R$ and $[X_2^*,Y_1^*]_R$ are Hamiltonian bi-connected. From the degree conditions, we know that removing a small constant number of vertices will leave a nearly spanning subgraph which is Hamiltonian bi-connected.} and $|X_1^*| = |Y_2^*| \geq \frac{n}{2} - \frac{\gamma}{4}n$. ($X_i^*$ and $Y_i^*$ exist since $\frac{n}{2} - 20 \eta n - 4\theta n \geq \frac{n}{2} - \frac{\gamma}{4} n$.) 

If there were two disjoint red paths from $X_1^* \cup Y_2^*$ to $X_2^* \cup Y_1^*$, we could construct a nearly spanning red cycle\footnote{For the rest of the proof, we say that a path/cycle is nearly spanning if it has length $(2-o(1))n$. We do not carefully calculate the constants since we are only trying to construct a path/cycle of roughly half that size.}; so suppose there is at most one such red path.  
Let $W_R$ be the smallest set such that there are no red paths from $X_1^* \cup Y_2^*$ to $X_2^* \cup Y_1^*$ in $G-W_R$.  Note that by Menger's theorem, $W_R$ is either empty or consists of a single vertex which we denote $w_R$.  Without loss of generality, suppose $W_R\subseteq Y$.  Define a partition $\{\hat{X}_1, \hat{X}_2\}$ of $X$ and a partition $\{\hat{Y}_1, \hat{Y}_2\}$ of $Y\setminus W_R$ such that $X_i^*\subseteq \hat{X}_i$ and $Y_i^*\setminus W_R\subseteq \hat{Y}_i$ for all $i\in [2]$ and there are no red edges between $\hat{X}_1 \cup \hat{Y}_2$ and $\hat{X}_2 \cup \hat{Y}_1$.

If $|\hat{X}_i|\geq \frac{n}{2}$ and $|\hat{Y}_{i}| \geq \frac{n}{2}$ for some $i\in [2]$, then we can find a blue cycle of length at least $2 \lceil n/2 \rceil$. (Note that $[\hat{X}_i, \hat{Y}_{i}]$ is completely blue and has minimum degree $n/4 + \gamma n$, so by Theorem \ref{berge_lemma}, $[\hat{X}_i, \hat{Y}_{i}]_B$ is Hamiltonian bi-connected.)

If $|\hat{X}_1| = \frac{n}{2}$, then $|\hat{X}_2|=\frac{n}{2}$ and $n$ must be even and so at least one of $\hat{Y}_1$ and $\hat{Y_2}$ must have order at least $\ceiling{\frac{n-1}{2}}=\frac{n}{2}$ and we are done by the previous paragraph.  
So without loss of generality, suppose 
\begin{equation*}
|\hat{X}_1| \geq \frac{n+1}{2} ~\text{ and }~ |\hat{Y}_1| \leq \frac{n-1}{2}. 
\end{equation*}

If there is a blue matching of size 2 in $[\hat{X}_1, \hat{Y}_2]$, then we can find a nearly spanning blue cycle; so suppose not.  

If there are no blue edges in $[\hat{X}_1, \hat{Y}_2]$, then if $|\hat{Y}_2|\geq n/2$ we can find a red cycle of length at least $2\ceiling{n/2}$ in $[\hat{X}_1, \hat{Y}_2]_R$; so suppose $|\hat{Y}_2|<n/2$.  This is only possible if $n$ is odd, $|\hat{Y}_2|=\frac{n-1}{2}=|\hat{Y}_1|$ and $|W_R|=1$.  Now, by Theorem \ref{berge_lemma}, if $w_R$ has at least $n/8$ red edges to $\hat{X}_1$, then we can find a red cycle of length $2\ceiling{n/2}$ in $[\hat{X}_1, \hat{Y}_2 \cup \{w_R\}]_R$, or else $w_R$ has at least $n/8$ blue edges to $\hat{X}_1$ and we can find a red cycle of length $2\ceiling{n/2}$ in $[\hat{X}_1, \hat{Y}_1\cup \{w_R\}]_B$. 

So suppose finally that the size of a maximum blue matching in $[\hat{X}_1, \hat{Y}_1]$ is exactly 1.  First note that in this case there is a nearly spanning blue path.  So to complete the proof, we must find a monochromatic cycle of length at least $2\floor{n/2}$.  By K\"onig's theorem, there is a single vertex in $[\hat{X}_1, \hat{Y}_2]$ which is incident with all of the blue edges.  Suppose first that $v^*\in \hat{Y}_2$ is such a vertex.  Move $v^*$ to $\hat{Y}_1$ (that is, formally redefine $\hat{Y}_2:=\hat{Y}_2\setminus \{v^*\}$ and $\hat{Y}_1:=\hat{Y}_1\cup \{v^*\}$) if and only if $v^*$ has fewer than $n/8$ red neighbors in $\hat{X}_1$. If $|\hat{Y}_2|\geq n/2$, then we are done by finding a red cycle of length at least $2\ceiling{n/2}$ in $[X_1, Y_2]$; so suppose $|\hat{Y}_2|<n/2$.  This is only possible if $n$ is odd, $|\hat{Y}_2|=\frac{n-1}{2}=|\hat{Y}_1|$ and $|W_R|=1$.  Now if $w_R$ has at least $n/8$ red edges to $\hat{X}_1$, we can find a red cycle of length $2\ceiling{n/2}$ in $[\hat{X}_1, \hat{Y}_2\cup \{w_R\}]_R$, or else $w_R$ has at least $n/8$ blue edges to $\hat{X}_1$ and we can find a blue cycle of length $2\ceiling{n/2}$ in $[\hat{X}_1, \hat{Y}_1\cup \{w_R\}]_B$.  Suppose instead that $u^*\in \hat{X}_1$ is the vertex in $[\hat{X}_1, \hat{Y}_2]$, guaranteed by K\"onig's theorem, which is incident with all of the blue edges.  Move $u^*$ to $\hat{X}_2$ (that is, formally redefine $\hat{X}_1:=\hat{X}_1\setminus \{u^*\}$ and $\hat{X}_2:=\hat{X}_2\cup \{u^*\}$) if and only if $u^*$ has fewer than $n/8$ blue neighbors in $\hat{Y}_2$. If $|\hat{Y}_2|\geq n/2$, then either $|\hat{X}_1|\geq n/2$ and we are done by finding a red cycle in $[\hat{X}_1, \hat{Y}_2]_R$, or $|\hat{X}_2|\geq n/2$ and we are done by finding a blue cycle in $[\hat{X}_2, \hat{Y}_2]_B$. So suppose $|\hat{Y}_2|< n/2$.  This is only possible if $n$ is odd, $|\hat{Y}_2|=\frac{n-1}{2}=|\hat{Y}_1|$ and $|W_R|=1$. Since $n$ is odd and we had $|\hat{X}_1|>n/2$ before potentially moving $u^*$, we now have $|\hat{X}_1|\geq \frac{n-1}{2}$ and thus there is a red cycle of length at least $2\floor{n/2}$ in $[\hat{X}_1, \hat{Y}_2]_R$.

\end{proof}

\end{document}